\documentclass[a4paper,10pt]{amsart}

\usepackage{amsmath}
\usepackage{amssymb}
\usepackage{amsthm}
\usepackage{hyperref}

\newcommand{\Z}{\ensuremath{\mathbb{Z}}}
\newcommand{\A}{\ensuremath{\mathbb{A}}}

\newcommand{\N}{\ensuremath{\mathbb{N}}}

\newcommand{\ra}{\ensuremath{\rightarrow}}

\newcommand{\coker}{\ensuremath{\textnormal{coker}}}

\newcommand{\lm}{\ensuremath{\textnormal{lm}}}
\newcommand{\lt}{\ensuremath{\textnormal{lt}}}

\newcommand{\wt}{\ensuremath{\textnormal{wt}}}
\newcommand{\ord}{\ensuremath{\textnormal{ord}}}
\DeclareMathOperator{\HP}{HP}

\newcommand{\pe}{\ensuremath{\mathfrak{p}}}

\DeclareMathOperator{\Spec}{Spec}
\DeclareMathOperator{\Spf}{Spf}
\DeclareMathOperator{\Hom}{Hom}

\newcommand{\Hh}{\ensuremath{\mathbb{H}}}

\newtheorem{Thm}{Theorem}[section]
\newtheorem*{Thm*}{Theorem}
\newtheorem{Def}[Thm]{Definition}
\newtheorem{Prop}[Thm]{Proposition}
\newtheorem*{Prop*}{Proposition}
\newtheorem{Lem}[Thm]{Lemma}
\newtheorem*{Lem*}{Lemma}
\newtheorem{Cor}[Thm]{Corollary}
\newtheorem*{Cor*}{Corollary}

\newtheorem{Conj*}{Conjecture}

\theoremstyle{remark}
\newtheorem*{Rem}{Remark}

\newtheorem*{Ex*}{Example}

\title{Arc Spaces and Rogers-Ramanujan Identities}

\author[Clemens Bruschek]{Clemens Bruschek$^*$}
\address{University of Vienna, Nordbergstr. 15, 1090 Vienna, Austria}
\email{Clemens.Bruschek@univie.ac.at}
\thanks{$^*$FWF project P21461 and I382}

\author{Hussein Mourtada}
\address{Universit\'{e} de Versailles Saint-Quentin-en-Yvelines, 45
  Avenue des \'{E}tats-Unis, 78035, Cedex Versailles}
\email{Hussein.Mourtada@math.uvsq.fr}
\thanks{}

\author[Jan Schepers]{Jan Schepers$^{\dagger}$}
\address{K.U.Leuven, Celestijnenlaan 200B, 3001 Leuven, Belgium}
\email{janschepers1@gmail.com}
\thanks{$^{\dagger}$Postdoctoral Fellow of the Research Foundation - Flanders (FWO)\\ \indent \textit{MSC2010 Subject Classification:} 14B05, 11P84, 05A17, 13P10.}

\begin{document}


\begin{abstract}
  Arc spaces have been introduced in algebraic geometry as a tool to study
  singularities but they show strong connections with combinatorics as
  well. Exploiting these relations we obtain a new approach to the
  classical Rogers-Ramanujan Identities. The linking object is the
  Hilbert-Poincar\'{e} series of the arc space over a point of the
  base variety. In the case of the double point this is precisely the
  generating series for the integer partitions without equal or
  consecutive parts.
\end{abstract}

\maketitle

\normalsize
\section{Introduction}\label{sec:section_1}

Arc spaces describe formal power series solutions (in one variable) to 
polynomial equations. They first appeared in the work of Nash (published later as \cite{nash}), who investigated their relation to some intrinsic data of a resolution of singularities of a fixed algebraic variety. He asked whether there is a bijection between the irreducible components of the arc space based at the singular locus and the set of essential divisors. While this so-called `Nash problem' is still actively studied (see for instance the recent papers \cite{monique_ana,plenat_spivakovsky,pe,fernandez} and the overview \cite{ishii_intro}), in the last decade arc spaces have gained much interest from algebraic geometers, through their role in motivic integration and their utility in birational geometry.
Arc spaces show strong relations with combinatorics as well. In the present text we
indicate how to exploit this connection both for geometric as well as
combinatorial benefit. In particular, we expose a surprising connection with the well-known
Rogers-Ramanujan identities.\\

Let us emphasize the main algebraic and combinatorial aspects
presented here. First, we suggest to study local algebro-geometric
properties of algebraic (or analytic) varieties via natural
Hilbert-Poincar\'{e} series attached to arc spaces. In contrast to
already existing such series this one is sensitive to the
non-reduced structure of the arc space. Second, we propose
to derive identities between partitions by looking at suitable ideals
in a polynomial ring in countably many variables endowed with a
natural grading. Connecting both ideas will demand handling
Gr\"{o}bner basis in countably many variables, a problem which has
been successfully dealt with in different contexts over the last
years (see \cite{hillar_sullivant,draisma}). In the present situation -- that is for very specific ideals --
salvation from the natural obstruction of being infinitely generated
comes in the shape of a derivation making the respective ideals
differential.\\

We briefly indicate the connection between arc spaces and
partitions. Let $f\in k[x_1,\ldots, x_n]$ be a polynomial in $n$ variables $x_1,\ldots, x_n$
with coefficients in a field $k$. We denote the formal power series ring in one
variable $t$ over the field $k$ by $k[[t]]$. The \emph{arc
 space} $X_\infty$ of the algebraic variety $X$ defined by $f$ is the set of
power series solutions \begin{math} x(t)=(x_1(t),\ldots, x_n(t))\in k[[t]]^n
\end{math} to the equation $f(x(t))=0$. This set
turns out to be eventually algebraic in the sense that it is given by
polynomial equations (though there are countably many of
them). Indeed, expanding $f(x(t))$ as a power series in $t$
gives 
\begin{displaymath}
  f(x(t)) = F_0 + F_1t + F_2t^2 + \cdots
\end{displaymath}
where the $F_i$ are polynomials in the coefficients of $t$ in $x(t)$. Therefore, a given vector of formal
power series $a(t)\in k[[t]]^n$ is an element of the arc space
$X_\infty$ if and only if its coefficients fulfill
the equations $F_0,F_1,\dots$. Algebraically the corresponding set of
solutions is described by its \emph{coordinate algebra}
\begin{displaymath}
  J_\infty(X)= k[x_j^{(i)}; 1\leq j \leq n, i\in \N]/(F_0,F_1,\ldots),
\end{displaymath}
where $\mathbb{N}=\{0,1,2,\ldots \}$. The variable $x_j^{(i)}$ corresponds to the coefficient of $t^i$
in $x_j(t)$. We will mostly be interested in the case where $a(0)$
is a point on $X$ (without loss of generality we may assume that this is the origin). The resulting algebra, obtained from $J_\infty(X)$
by substituting $x_j^{(0)}=0$, is called the \emph{focussed
  arc algebra} and denoted by $J_\infty^0(X)$; we write $f_i$ for
the image of $F_i$ under this substitution:
\begin{displaymath}
  J_\infty^0(X)=k[x_j^{(i)}; 1\leq j \leq n, i\geq 1]/(f_1,f_2,\ldots).
\end{displaymath}
This algebra is naturally graded by the
weight function $\wt\, x_j^{(i)}=i$ since $f_{\ell}$ is homogeneous of weight $\ell$. In
the special case of $n=1$ we will write $y_i$ instead of
$x_1^{(i)}$. Integer partitions arise naturally when computing weights
of monomials in $J^0_\infty(\A^1)$. Recall that a \emph{partition} of
$m\in \N$ is an $r$-tuple of positive integers
$\lambda_1 \leq \lambda_2\leq \cdots \leq \lambda_r$ with $\lambda_1 +
\cdots + \lambda_r =m$. The $\lambda_i$ are the \emph{parts of the
  partition} and $r$ is its \emph{length}. A monomial
$y_1^{\alpha_1}\cdots y_e^{\alpha_e}$ has
weight $\alpha_1\cdot 1 + \cdots + \alpha_e\cdot
e$. Asking for the number of monomials (up to coefficients) of some
weight $m$ is thus asking for the number of partitions of $m$. This is 
precisely what we capture when computing the Hilbert-Poincar\'{e} series of
$J_\infty^0(\A^1)$. In general, the Hilbert-Poincar\'{e} series of $J_{\infty}^0(X)$ is defined as
\begin{displaymath}
  \HP_{J_\infty^0(X)}(t) = \sum_{j=0}^\infty \dim_k
  \left(J_\infty^0(X) \right)_j\cdot t^j,
\end{displaymath}
where $\left(J_\infty^0(X) \right)_j$ denotes the $j$th homogeneous
component of $J_\infty^0(X)$. In the simple case of $X=\A^1$ we may use the generating function for partitions to represent
$\HP_{J_\infty^0(\A^1)}(t)$ by 
\begin{displaymath}
  \mathbb{H} := \prod_{i\geq 1} \frac{1}{1-t^i}.
\end{displaymath}  

By the general theory of Gr\"{o}bner basis $\HP_{J_\infty^0(X)}(t)$ is identical with the Hilbert-Poincar\'{e} series of the algebra 
\begin{displaymath}
  k[x_j^{(i)}; 1\leq j \leq n, i\geq 1]/L(I),
\end{displaymath}
where $L(I)$ denotes the \emph{leading ideal} of $I=(f_0,f_1,\ldots)$ (with respect to 
a chosen monomial ordering). The leading ideal is much simpler since
it is generated by monomials. Computing the Hilbert-Poincar\'{e}
series of the respective algebra corresponds to counting partitions, leaving out
those coming from weights of monomials in $L(I)$. In the simple
example of $f=y^2$, these will be all partitions without repeated or
consecutive parts. Such partitions are part of the well-known
\emph{Rogers-Ramanujan identity}: the number of partitions of $n$ into
parts congruent to 1 or 4 modulo 5 is equal to the number of
partitions of $n$ into parts that are neither repeated nor
consecutive (see \cite{andrews_partitions}). This gives:

\begin{Thm*}
Let $k$ be field of characteristic $0$. For $X\colon y^2=0$ we compute
\begin{displaymath}
 \HP_{J_{\infty}^{0}(X)}(t) = \prod_{\substack{i\geq 1 \\ i\equiv
    1,4 \bmod 5}} \frac{1}{1-t^i}.
\end{displaymath}
\end{Thm*}

More generally, we obtain using Gordon's generalizations of the
Rogers-Ramanujan identities for $X\colon y^n=0$, $n\geq 2$:

\begin{samepage}
\begin{Thm*}
\begin{displaymath}
\HP_{J_\infty^0(X)}(t)=\Hh\cdot
  \prod_{\substack{i\geq 1\\ i \equiv 0, n,n+1 \\ \bmod 2n+1}} (1-t^i).
\end{displaymath}
\end{Thm*}
\end{samepage}

The computation of $L(I)$ in these rather simple
looking cases is nontrivial. It is carried out in Section
\ref{sec:section_5}.\\

Moreover, standard techniques from commutative algebra allow to compute a recursion for the
Hilbert-Poincar\'{e} series in the case of $X\colon y^2=0$:

\begin{Prop*}
The generating series $\HP_{J_{\infty}^{0}(X)}(t)$ is the $t$-adic limit of the sequence of formal power series $A_d$ defined by:
$$A_1=A_2=1, \quad \text{and} \quad A_d = A_{d-1}+t^{d-2} A_{d-2}\text{ for } d\geq 3.$$
\end{Prop*}

The above proposition was first found in an empirical way in
\cite{andrews_baxter}, and it leads to the Rogers-Ramanujan
identities (see Section \ref{sec:section_5}).\\

Returning to the geometric aspect of the series attached to arc spaces, we
compute them in other interesting cases: for smooth
points, for rational double points of surfaces, and for normal
crossings singularities. In these cases the simple geometry of the jet
schemes permits to compute the Hilbert-Poincar\'{e} series defined
above, and we find the following (see
Propositions~\ref{prop:smooth_case}, \ref{cor:rational_double_point}
and \ref{prop:normal_crossings}).

\begin{Prop*} With the above introduced notation, we have:
\begin{itemize}
\item[(1)] If $\mathfrak{p}$ is a smooth point on a variety $X$ of dimension $d$, then 
\[ \HP_{J_{\infty}^{\mathfrak{p}}(X)} = \left(\prod_{i\geq 1} \frac{1}{1-t^i}\right)^d, \]
\item[(2)] if $X$ is a surface with a rational double point at $\mathfrak{p}$ then 
\[ \HP_{J_{\infty}^{\mathfrak{p}}(X)}(t) =  \left(\frac{1}{1-t}\right)^{3} \left(\prod_{i\geq 2} \frac{1}{1-t^i}\right)^2, \]
\item[(3)] if $X= \{ x_1\cdots x_{d+1} = 0\} \subset \mathbb{A}^{d+1}_k$, and $\mathfrak{p}$ is a point in the intersection of precisely $e$ components, then 
\[ \HP_{J_{\infty}^{\mathfrak{p}}(X)}(t) =  \left(\prod_{i=1}^{e-1} \frac{1}{1-t^i}\right)^{d+1} \left(\prod_{i\geq e} \frac{1}{1-t^i}\right)^d. \]
\end{itemize}
\end{Prop*}

In the past several generating series have been associated to
the arc space of a (singular) variety by Denef and Loeser (see
\cite{denef_loeser,denef_loeser_definable_sets}), in analogy with the
$p$-adic case. All those series are defined in the motivic setting and
take values in a power series ring with coefficients in (a
localization of) a Grothendieck ring. They are rational. Some of those
series encode more information about the singularities than
others. For a comparison between them see
\cite{nicaise_manuscripta,nicaise,helena_pedro}. While those series
are concerned with the reduced structure of the arc space, our series
is \emph{sensitive to the non-reduced structure} as well, although it is not
rational in general. We discuss more geometric motivation for introducing the above Hilbert-Poincar\'{e} series in Section~\ref{sec:section_3} (after Definition~\ref{def:arc_Hilbert}).\\

The structure of the paper is as follows: in Section~\ref{sec:section_2}, we define jet schemes and arc spaces and recall some basic facts about them.
In Section~\ref{sec:section_3} we introduce the arc
Hilbert-Poincar\'{e} series, we discuss its properties and we compute
it in particular cases. Section~\ref{sec:section_4} recalls basic
facts about partitions and the Rogers-Ramanujan
identities. Section~\ref{sec:section_5} is devoted to the main
theorem. For the convenience of the reader, we recall the facts about
Hilbert-Poincar\'{e} series and Gr\"obner bases that we use in the
paper in an appendix (Section~\ref{appendix}).

\subsection*{Acknowledgements} 
The authors thank the organizers of YMIS for preparing the ground for
this joint research. The first named author expresses his gratitude to
Herwig Hauser, Georg Regensburger and Josef Schicho for many useful
discussions on these topics; the second named author is grateful to
Monique Lejeune-Jalabert for introducing him to Hilbert-Poincar\'{e}
series; the third named author thanks Johannes Nicaise for explaining 
him several facts about arc spaces in detail and he thanks the
Institut des Hautes \'{E}tudes Scientifiques, where part of this work
was done, for hospitality.\\

\section{Jet schemes and arc spaces}\label{sec:section_2}

Let $k$ be a field. Let $X=\Spec \bigl(k[x_1, \ldots ,x_n] / (g_1, \ldots ,g_r)\bigr)$ be an affine scheme of finite type over $k$. For $l\in \{1, \ldots ,r\}$, and $j\in \{0, \ldots ,m\}$, we define the polynomial $G_l^{(j)} \in k[x_s^{(i)}; 1\leq s\leq n, 0\leq i\leq m]$ as the coefficient of $t^j$ in the expansion of
\begin{equation}
  g_l( x_1^{(0)}+x_1^{(1)}t+  \cdots  + x_1^{(m)} t^m, \ldots, x_n^{(0)}+x_n^{(1)} t+  \cdots  + x_n^{(m)}t^m ) . 
 \end{equation}
Then the \textit{$m$th jet scheme} $X_m$ of $X$ is
\[ X_m := \Spec\left(\frac{k[x_s^{(i)}; 1\leq s\leq n, 0\leq i\leq m]}{(G_l^{(j)}; 1\leq l \leq r , 0\leq j \leq m) } \right). \]
In particular, we have that $X_0=X$. Of course, we do not need to fix $m$ in advance, and we can define $G_l^{(j)} \in k[x_s^{(i)}; 1\leq s\leq n, i\in \mathbb{N}]$ for all $j\in \mathbb{N}$ as above. Here $\mathbb{N} = \{0,1,\ldots\}$. Then the \textit{arc space} $X_{\infty}$ of $X$ is
\[ X_{\infty} := \Spec\left(\frac{k[x_s^{(i)}; 1\leq s\leq n, i\in \mathbb{N}]}{(G_l^{(j)}; 1\leq l \leq r , j\in \mathbb{N}) } \right). \]
The functorial definition is also useful. Let $X$ be a scheme of finite type over $k$ and let $m \in \mathbb{N}$. The functor 
\[ \mathcal{F}_m : k\text{-Schemes} \rightarrow \text{Sets} \]
which to an affine scheme defined by a $k$-algebra $A$ associates 
\[ \mathcal{F}_m(\Spec(A))= \Hom_k\bigr(\Spec\bigl(A[t]/(t^{m+1}) \bigr) ,X\bigr) \]
is representable by the $k$-scheme $X_m$ (see for example \cite{ishii_intro,vojta}). The arc space $X_{\infty}$ represents the functor $\mathcal{F}_{\infty}$ that associates to a $k$-algebra $A$ the set $\Hom_k\bigr( \Spf(A[[t]]) ,X\bigr)$, where $\Spf$ denotes the formal spectrum. 

For $m,p \in \mathbb{N}, m > p$, the truncation homomorphism $A[t]/(t^{m+1}) \rightarrow A[t]/(t^{p+1})$ induces
a canonical projection $\pi_{m,p}: X_m \rightarrow X_p.$ These morphisms clearly verify $\pi_{m,p}\circ \pi_{q,m}=\pi_{q,p}$
for $p<m<q$. We denote the canonical projection $\pi_{m,0}:X_m\rightarrow X_0$ by $\pi_{m}$. For $m\in \mathbb{N}$ we also have the truncation morphism $A[[t]] \rightarrow A[t]/(t^{m+1})$. It gives rise to a canonical morphism $\psi_{m}: X_{\infty} \longrightarrow X_m$.

Assume now that $k$ has characteristic zero. In that case we can explicitly determine the ideals defining the jet schemes and the arc space. Let $S=k[x_1, \ldots ,x_n]$ and
$S_{m}=k[x_s^{(i)}; 1\leq s \leq n, 0\leq i\leq m]$. Let $D$ be the $k$-derivation on $S_m$ defined by $D(x_s^{(i)}):= x_s^{(i+1)}$ if $0\leq i <m$, and $D(x_s^{(m)}):= 0$. We embed $S$ in $S_m$ by mapping $x_i$ to $x_i^{(0)}$.

\begin{Prop}\label{prop:equations_by_deriving} 
Let $X=\Spec\bigl(S/(g_1, \ldots ,g_r)\bigr)$ and let $J_m(X)$ be the coordinate ring of $X_m$. Then
$$J_m(X) = \frac{S_m}{(D^j(g_l) ; 1\leq l\leq r, 0\leq j\leq m)}.$$
\end{Prop}
 
\begin{proof}
Since $k$ has characteristic zero, we may equally well replace $x_i$ by
\[ \frac{x_i^{(0)}}{0!} + \frac{x_i^{(1)}}{1!} t + \cdots + \frac{x_i^{(m)}}{m!} t^m \]
to obtain the equations of the jet space. For $g\in S$ we denote then
\[ \phi(g) := g\biggl( \frac{x_1^{(0)}}{0!} + \frac{x_1^{(1)}}{1!} t + \cdots + \frac{x_1^{(m)}}{m!} t^m ,\ldots , \frac{x_n^{(0)}}{0!} + \frac{x_n^{(1)}}{1!} t + \cdots + \frac{x_n^{(m)}}{m!} t^m \biggr).\]
Then we have
$$ \tau_m\bigl(\phi(g)\bigr) = \sum_{j=0}^m \frac{D^j(g)}{j!}\, t^j,$$
where $\tau_m$ means truncation at degree $m$. To see this, it is sufficient to remark that it is true for $g=x_i$, and that both sides of the
equality are additive and multiplicative in $g$ (after truncating at degree $m$). The proposition follows.
\end{proof}

Similarly, the coordinate ring $J_{\infty}(X)$ of $X_{\infty}$ is given by 
\[ J_{\infty}(X) = \frac{k[x_s^{(i)}; 1\leq s \leq n, i\in \mathbb{N}]}{(D^j(g_l) ; 1\leq l\leq r, j\in \mathbb{N})} .\] 
Here $D(x_s^{(i)}) = x_s^{(i+1)}$ for all $i\in \mathbb{N}$. For further understanding of the equations of the jet schemes and their relation with Bell polynomials, see \cite{bruschek_thesis,bruschek_JetAndBell}.

\section{The arc Hilbert-Poincar\'{e} series}\label{sec:section_3}

In this section we introduce and discuss the Hilbert-Poincar\'{e} series of the arc algebra of a (not necessarily reduced nor irreducible) algebraic variety $X$, focussed at a point $\mathfrak{p}$ of $X$. Since this will be a local invariant, we may restrict ourselves to $X$ being a closed subscheme of affine space. For generalities about Hilbert-Poincar\'{e} series of graded algebras we refer to the appendix (Section~\ref{appendix}).\\

As above, let $k$ be a field of characteristic zero. Although for most of the statements it is not necessary that $k$ is algebraically closed, we will assume it for convenience. Let $X$ be a subscheme of affine $n$-space over $k$, defined by some ideal $I$ in $k[x_1,\ldots , x_n]$. We define a grading on the polynomial ring $k[x_j^{(i)};1\leq j \leq n, i\in  \mathbb{N}]$ by putting the weight of $x_j^{(i)}$ equal to $i$. We prefer to use the terminology `weight' instead of `degree' here, in order not to confuse with the usual degree. It is easy to see that the ideal $I_{\infty}$ of $k[x_j^{(i)};1\leq j \leq n, i\in  \mathbb{N}]$ defining the arc space $X_{\infty}$ is homogeneous (with respect to the weight) and hence the arc algebra $J_{\infty}(X)$ is graded as well (this follows for instance from Proposition~\ref{prop:equations_by_deriving}). Similarly the jet algebras $J_m(X)$ are graded. Let $\mathfrak{p}$ be any point of $X$ and denote by $\kappa(\mathfrak{p})$ the residue field at $\mathfrak{p}$. After identifying $J_0(X)$ with the coordinate ring of $X$ we obtain a natural map from $J_0(X)$ to $\kappa(\mathfrak{p})$. 

\begin{Def}\label{def:arc_Hilbert}
We define the \emph{focussed arc algebra 
of $X$ at $\mathfrak{p}$} as
\[ J_{\infty}(X) \otimes_{J_0(X)} \kappa(\mathfrak{p}) \]
and we denote it by $J_{\infty}^{\mathfrak{p}}(X)$. Analogously, we define the \emph{focussed jet algebras of $X$ at $\mathfrak{p}$} by
\[ J_{m}^{\mathfrak{p}}(X):= J_{m}(X) \otimes_{J_0(X)} \kappa(\mathfrak{p}). \]
Using the above grading, we write $ \HP_{J_{\infty}^{\mathfrak{p}}(X)}(t)$ respectively $\HP_{J_{m}^{\mathfrak{p}}(X)}(t)$ for their Hilbert-Poincar\'{e} series as graded $\kappa(\mathfrak{p})$-algebras. We call this the \emph{arc Hilbert-Poin\-ca\-r\'{e} series at $\mathfrak{p}$} respectively the \emph{mth jet Hilbert-Poincar\'{e} series at $\mathfrak{p}$}.
\end{Def}

In fact, the focussed arc algebra at a point $\mathfrak{p}$ is the coordinate ring of the scheme theoretic fiber of the morphism $\psi_0 : X_{\infty} \to X$ over $\mathfrak{p}$. Note that the weight zero part of $J_{\infty}^{\mathfrak{p}}(X)$ or $J_{m}^{\mathfrak{p}}(X)$ is always a one-dimensional $\kappa(\mathfrak{p})$-vector space. \\

In the special case that $X$ is a hypersurface given by a polynomial $F\in k[x_1,\ldots,x_n]$ with $F(0)=0$, this boils down to the following. We define $F_0$ to be $F$ in the variables $x_1^{(0)},\ldots , x_n^{(0)}$. Then we put $F_1:= DF_0, F_2:=DF_1,\ldots $, where $D$ is the derivation from Section~\ref{sec:section_2}.
The arc algebra $J_{\infty}(X)$ is given as the quotient of $k[x_j^{(i)}; 1\leq j\leq n, i\in \mathbb{N}]$ by $(F_0,F_1,\ldots )$ (see Proposition~\ref{prop:equations_by_deriving}). And $J_{\infty}^0(X)$ is the quotient of $k[x_j^{(i)}; 1\leq j\leq n, i\geq 1]$ by $(f_0,f_1,\ldots )$, where $f_i$ is $F_i$ evaluated in $x_1^{(0)}= \cdots = x_n^{(0)} = 0$.

\begin{Rem}
Besides that the arc Hilbert-Poincar\'{e} series is very natural to look
at when working with arc spaces, its introduction is motivated by the
following. If $X$ is for instance a hypersurface then the ideal $I_m^0 = (f_0,f_1,\ldots )$ defining the fiber over the origin of the $m$th jet space of $X$ is generated by polynomials depending only on a subset of the variables
of the polynomial ring $k[x_j^{(i)}; 1\leq j\leq n, i\geq 1]$. Heuristically,
for a given $m \in \mathbb{N}$, the more $X$ is singular, the less
variables appear in $I_m^0$. So this series was meant as a
Hironaka type invariant of the singularity (see
\cite{berthomieu_hivert_mourtada}), that is a kind of measure of the
number of variables appearing in $I^{0}_m$.

Note also that since the jet spaces are far from being equidimensional
in general (see
\cite{mourtada_plane_branches,mourtada_toric_surfaces}), the jet
algebras have a big homological complexity, what makes it difficult to
compute the series introduced above.
\end{Rem}

\begin{Rem}\label{rem:graded_structure}
We can define the graded structure on $J_{\infty}(X)$ more intrinsically as follows, with $X$ as above. For an extension field $K$ of $k$, the $K$-rational points of $X_{\infty}$ correspond to morphisms of $k$-algebras
\[  \gamma : \Gamma(X,\mathcal{O}_X) \to K[[t]]. \]
For $\lambda\in k^{\times}$ we have an automorphism $\varphi_{\lambda}$ of $K[[t]]$ determined by $t\mapsto \lambda t$. By composing $\varphi_{\lambda}$ with $\gamma$ we obtain a natural action of $k^{\times}$ on $X_{\infty}$ and hence on its coordinate ring $J_{\infty}(X)$. An element $f\in J_{\infty}(X)$ is then called homogeneous of weight $i$ if $\lambda\cdot f = \lambda^i f$, for all $\lambda \in k^{\times}$.  
\end{Rem}

We consider the truncation operator
\[  \tau_{\leq r} : k[[t]] \to k[t] : \sum_{i\geq 0} a_it^i \mapsto \sum_{i=0}^r a_it^i.\] 
Then we have the following simple observation.

\begin{Prop}\label{prop:approximate_by_jet_spaces}
$\tau_{\leq m} \HP_{J_{m}^{\mathfrak{p}}(X)}(t) = \tau_{\leq m} \HP_{J_{\infty}^{\mathfrak{p}}(X)}(t).$
\end{Prop}

Now let $X$ and $Y$ be closed subschemes of $\mathbb{A}^n_k$ respectively $\mathbb{A}^m_k$. Recall that one calls $\mathfrak{p}\in X$ and $\mathfrak{q}\in Y$ \textit{analytically isomorphic} if $\widehat{\mathcal{O}}_{X,\mathfrak{p}}$ and $\widehat{\mathcal{O}}_{Y,\mathfrak{q}}$ are isomorphic $k$-algebras.

\begin{Prop}\label{prop:series_is_analytic_invariant}
If $\mathfrak{p}\in X$ and $\mathfrak{q}\in Y$ are analytically isomorphic then 
\[ \HP_{J_{\infty}^{\mathfrak{p}}(X)}(t) =  \HP_{J_{\infty}^{\mathfrak{q}}(Y)}(t).\]
\end{Prop}

\begin{proof}
The fiber of $\pi_X : X_{\infty} \to X$ over $\pe$ is a scheme over $\kappa(\mathfrak{p})$ whose $K$-rational points, for an extension field $K$ of $\kappa(\mathfrak{p})$, correspond to the set of morphisms of $k$-algebras
\[ \gamma: \Gamma(X,\mathcal{O}_X) \to K[[t]]          \]
such that $\gamma^{-1}\bigl((t)\bigr) = \mathfrak{p}$. Since $K[[t]]$ is complete, $\gamma$ factors uniquely through $\widehat{\mathcal{O}}_{X,\mathfrak{p}}$. We conclude that the fiber of $\pi_X : X_{\infty} \to X$ over $\pe$ is determined by $\widehat{\mathcal{O}}_{X,\mathfrak{p}}$. 

To see that the graded structure of the fibers of $\pi_X$ and $\pi_Y$ above $\mathfrak{p}$ respectively $\mathfrak{q}$ agree, we can use the intrinsic description of the graded structure from the previous remark.
\end{proof}

Next we compute the arc Hilbert-Poincar\'{e} series at a smooth point. We use the notation
\[  \mathbb{H} := \prod_{i\geq 1} \frac{1}{1-t^i}.    \]

\begin{Prop}\label{prop:smooth_case}
Let $X$ be an irreducible closed subscheme of $\mathbb{A}^n_k$ of dimension $d$ and let $\mathfrak{p}\in X$ be a smooth point. Then
\[  \HP_{J_{\infty}^{\mathfrak{p}}(X)}(t) = \mathbb{H}^{d}. \]
\end{Prop}

\begin{proof}
By definition $\mathcal{O}_{X,\mathfrak{p}}$ is a regular local ring, and hence an integral domain. Therefore $X$ is reduced, and thus an integral scheme.
Denote by $e$ the transcendence degree of $\kappa(\mathfrak{p})$ over $k$. From dimension theory (e.g.\ Thm.\ A and Cor.\ 13.4 on p.290 in \cite{eisenbud_commalg}) it follows that the dimension of $\mathcal{O}_{X,\mathfrak{p}}$ equals $d-e$. The complete local ring $\widehat{O}_{X,\mathfrak{p}}$ is regular as well and hence isomorphic to $\kappa(\mathfrak{p})[[x_1,\ldots , x_{d-e}]]$ (Prop.\ 10.16 in \cite{eisenbud_commalg}). By the theorem of the primitive element, $\kappa(\mathfrak{p})$ is isomorphic to $k(y_1,\ldots,y_e)[x]/(f)$, where we may assume $f\in k[y_1,\ldots,y_e,x]$. Hence $\kappa(\mathfrak{p})$ is isomorphic to the residue field of the point $(f,z_1,\ldots , z_{d-e-1})$ in the affine space 
\[ \Spec k[y_1,\ldots,y_e,x,z_1,\ldots, z_{d-e-1}].     \]
From Proposition~\ref{prop:series_is_analytic_invariant} it follows that it suffices to compute the arc Hilbert-Poincar\'{e} series at a point $\mathfrak{q}$ of $Y:=\mathbb{A}^{d}_{k}$. This is an easy task, since $J_{m}(Y)$ is the polynomial ring
\[ k[x_j^{(i)}; 1\leq j \leq d, 0\leq i\leq m],\] 
and $J_{m}^{\mathfrak{q}}(Y)$ equals
\[ \kappa(\mathfrak{q})[x_j^{(i)}; 1\leq j \leq d, 1\leq i\leq m].\] 
The variables form a regular sequence in this ring, and hence it follows easily from Lemma~\ref{lem:hilbert_trick_1} that
\[ \HP_{J_{m}^{\mathfrak{q}}(Y)}(t) = \left(\prod_{i=1}^m \frac{1}{1-t^i} \right)^{d}.\]
Now we use Proposition~\ref{prop:approximate_by_jet_spaces} to finish the proof.
\end{proof}

In Section~\ref{sec:section_4} we discuss the connection of this result with partitions. We leave the proof of the following proposition to the reader.

\begin{Prop}\label{prop:multiplicativity_for_products}
Let $X$ and $Y$ be closed subschemes of $\mathbb{A}^n_k$ respectively $\mathbb{A}^m_k$. Let $\mathfrak{p}\in X$ and $\mathfrak{q}\in Y$. Then
\[ \HP_{J_{\infty}^{(\mathfrak{p},\mathfrak{q})}(X\times Y)}(t) = \HP_{J_{\infty}^{\mathfrak{p}}(X)}(t)\cdot \HP_{J_{\infty}^{\mathfrak{q}}(Y)}(t).  \]
\end{Prop}

The multiplicity of a singular point on a hypersurface can be easily read from the arc Hilbert-Poincar\'{e} series, as the reader may convince himself of:

\begin{Prop}\label{prop:multiplicity_from_HPseries}
Let $X$ be a hypersurface in $\mathbb{A}^n_k$ defined by a polynomial $F\in k[x_1,\ldots,x_n]$ with $F(0)=0$. Then $X$ has multiplicity $r$ at the origin if and only if $r$ is the maximal number such that
\[  \tau_{\leq r-1} \HP_{J_{\infty}^{0}(X)}(t) = \tau_{\leq r-1} \mathbb{H}^n. \]
Moreover, $\tau_{\leq r} \HP_{J_{\infty}^{0}(X)}(t) = \tau_{\leq r} \mathbb{H}^n - t^r$.
\end{Prop}

Next we derive a formula for the arc Hilbert-Poincar\'{e} series of the focussed arc algebra at a canonical hypersurface singularity of maximal multiplicity. First we recall the definition of a canonical singularity. Let $X$ be a normal variety. Assume that $X$ is $\mathbb{Q}$-Gorenstein (i.e.\ $rK_X$ is Cartier for some $r\geq 1$). Let $f: Y\to X$ be a log resolution. This means that $f$ is a proper birational morphism from a smooth variety $Y$ such that the exceptional locus is a simple normal crossings divisor with irreducible components $E_i, i\in I$. We have a linear equivalence
\[  K_Y = f^*K_X + \sum_i a_i E_i \] 
for uniquely determined $a_i\in \mathbb{Q}$ (these are called discrepancy coefficients). Then $X$ has \textit{canonical singularities} if $a_i \geq 0$ for all $i$. We say that $X$ has a canonical singularity at a point $\mathfrak{p}\in X$ if there exists a neighbourhood $U$ of $\mathfrak{p}$ in $X$ with canonical singularities. Note that if $X$ is a hypersurface in $\mathbb{A}^n_k$, then $X$ is Gorenstein (i.e.\ $K_X$ is Cartier) and all $a_i$ are then integers. In that case, if $X$ has a canonical singularity at a closed point $\mathfrak{p}$, the multiplicity of $\mathfrak{p}$ is at most the dimension of $X$. This follows by computing the discrepancy coefficient of the exceptional divisor of the blowing-up in $\mathfrak{p}$. 

\begin{Prop}\label{prop:canonical_sing}
Let $X$ be a normal hypersurface in $\mathbb{A}^n_k$ with a canonical singularity of multiplicity $n-1$ at the origin. Then 
\[ \HP_{J_{\infty}^{0}(X)}(t) = \left(\prod_{i=1}^{n-2} \frac{1}{1-t^i}\right)^{n} \left(\prod_{i\geq n-1} \frac{1}{1-t^i}\right)^{n-1} .\]
\end{Prop}

\begin{proof}
Let $X$ be defined by the polynomial $F$. We use the notations $F_i$ and $f_i$ as before. Then $f_i = 0$ for $0\leq i \leq n-2$. To deduce the result, it suffices to show that for every $m\geq n-1$ the polynomials $f_{n-1}, f_n, \ldots , f_m$ form a regular sequence in the polynomial ring $k[x_j^{(i)};1\leq j \leq n, 1\leq i\leq m]$, in view of Lemma~\ref{lem:hilbert_trick_1} and Proposition~\ref{prop:approximate_by_jet_spaces}. 

Since the question is local, we may assume that all singularities of $X$ are canonical. We will use a theorem by Ein and Musta\c{t}\u{a} that characterizes canonical singularities by the fact that their jet spaces are irreducible (see Thm.\ 1.3 in \cite{ein_mustata}). It is well known that the natural maps $\pi_m : X_m \to X$ are locally trivial fibrations above the smooth part of $X$, with fiber isomorphic to $\mathbb{A}_k^{(n-1)m}$. Hence the dimension of $X_m$ is precisely $(n-1)(m+1)$. Since $X_m$ is irreducible, it follows that the fiber $\pi_m^{-1}(0)$ has dimension at most $(n-1)(m+1)-1$. From dimension theory (e.g.\ Cor.\ 13.4 in \cite{eisenbud_commalg}) we deduce that the codimension of the ideal $I_m^{0}:=(f_{n-1}, f_n, \ldots , f_m)$ in $A_m:=k[x_j^{(i)};1\leq j \leq n, 1\leq i\leq m]$ is at least $nm - \bigl((n-1)(m+1) - 1\bigr) = m-n+2$. From the principal ideal theorem (Thm.\ 10.2 of \cite{eisenbud_commalg}) we get then that the codimension of $I_m^{0}$ is precisely $m-n+2$. Since a polynomial ring over a field is Cohen-Macaulay, we may apply the unmixedness theorem (Cor.\ 18.14 of \cite{eisenbud_commalg}) to deduce that every associated prime of $I_m^{0}$ is minimal. But the codimension of $I_{m+1}^{0}$ in $A_{m+1}$ is at least $m-n+3$, so this means that $f_{m+1}$ is not contained in any minimal prime ideal containing $I_m^{0}$, considered as ideal of $A_{m+1}$. Hence $f_{m+1}$ does not belong to an associated prime ideal of $I_m^{0}$, and thus it is a nonzerodivisor modulo $I_m^{0}$ (see Thm.\ 3.1(b) of \cite{eisenbud_commalg}).
\end{proof}

It follows that the arc Hilbert-Poincar\'{e} series is in this case completely determined by the multiplicity. As a corollary, we get the following nice example. This result was obtained by explicit computation in \cite{mourtada_thesis}.

\begin{Cor}\label{cor:rational_double_point}
If $X$ is a surface with a rational double point at $\mathfrak{p}$ then 
\[ \HP_{J_{\infty}^{\mathfrak{p}}(X)}(t) =  \left(\frac{1}{1-t}\right)^{3} \left(\prod_{i\geq 2} \frac{1}{1-t^i}\right)^2.\]
\end{Cor}

A similar result is true for normal crossings singularities. A scheme $X$ of finite type over $k$ of dimension $d$ is said to have \textit{normal crossings} at a point $\mathfrak{p}$ if $\mathfrak{p}\in X$ is analytically isomorphic to a point $\mathfrak{q}\in Y$, where $Y$ is the hypersurface in $\mathbb{A}_k^{d+1}$ defined by $y_1\cdots y_{d+1} = 0$. For points on $Y$ the situation is as follows.

\begin{Prop}\label{prop:normal_crossings}
Let $Y$ be as above, and assume that $\mathfrak{q}$ lies precisely on the irreducible components given by $y_1=0,\ldots , y_e=0$. Then 
\[ \HP_{J_{\infty}^{\mathfrak{q}}(Y)}(t) =  \left(\prod_{i=1}^{e-1} \frac{1}{1-t^i}\right)^{d+1} \left(\prod_{i\geq e} \frac{1}{1-t^i}\right)^d .\]
\end{Prop}

\begin{proof}
Locally at $\mathfrak{q}$, the variety $Y$ looks like a product of the hypersurface $Z$ given by $z_1\cdots z_e =0$ in $\mathbb{A}^e_k$ and the affine space $\mathbb{A}^{d+1-e}_k$. By Propositions~\ref{prop:multiplicativity_for_products} and \ref{prop:smooth_case} it suffices now to show that 
\begin{equation}\label{formula_normal_crossings}  \HP_{J_{\infty}^{0}(Z)}(t)  =  \left(\prod_{i=1}^{e-1} \frac{1}{1-t^i}\right)^{e} \left(\prod_{i\geq e} \frac{1}{1-t^i}\right)^{e-1}.  \end{equation}
According to Theorem 2.2 of \cite{goward_smith}, the $m$th jet scheme $Z_m$ is equidimensional of dimension $(e-1)(m+1)$, and for $m+1\geq e$ there are actually irreducible components of that dimension in the fiber of $Z_m$ above the origin in $Z$. We can use a reasoning as in the proof of Proposition~\ref{prop:canonical_sing} to conclude that the $m-e+1$ equations $f_e,f_{e+1} , \ldots , f_m$ form a regular sequence in $k[z_j^{(i)}; 1\leq j \leq e, 1\leq i\leq m]$ and then we use Lemma~\ref{lem:hilbert_trick_1} once more to deduce formula~(\ref{formula_normal_crossings}).
\end{proof}

\section{Partitions and the Rogers-Ramanujan identities}\label{sec:section_4}

A {\it partition (of length $r$)} of a postive integer $n$ is a
non-decreasing sequence $\lambda=(\lambda_1,\ldots, \lambda_r)$ of
positive integers $\lambda_i$, $1\leq i \leq r$, such that $$\lambda_1+
\cdots + \lambda_r =n.$$ The integers $\lambda_i$ are called the {\it
  parts of the partition $\lambda$}. We will denote the number of partitions of
$n$ by $p(n)$, with $p(0):=1$. In the
following we collect a few facts about integer partitions which will
be used in the subsequent sections. For an introduction to this topic
we refer for example to \cite{wilf_lecturesonis}; an extensive
treatment can be found in \cite{andrews_partitions}.

\begin{Prop}\label{prop:generatingfunction_partitions}
  The generating series of the partition function $p$ has the
  following infinite product representation: $$\sum_{i=0}^\infty p(n)t^n = \prod_{i\geq
    1}\frac{1}{1-t^i}.$$
\end{Prop}

Note, that this is precisely the series $\Hh$ which we have obtained as the
Hilbert-Poincar\'{e} series of the graded algebra $k[y_1, y_2,
\ldots]$ where the grading is given by $\wt\, y_i =i$. More generally,
the arc Hilbert-Poincar\'{e} series of an $d$-dimensional variety at a smooth point was given by
$\Hh^d$. This leads us to expect a connection between the
Hilbert-Poincar\'{e} series of arc algebras and partitions.\\

The following result is known in the literature as the {\it (first)
 Rogers-Ramanujan identity}. For a classical proof and an account of
its history, see Chpt.\ 7 of \cite{andrews_partitions}.

\begin{Thm}[Rogers-Ramanujan identity]\label{thm:rogersramanujan}
  The number of partitions of $n$ into parts congruent to 1 or 4 modulo 5
  is equal to the number of partitions of $n$ into parts that are neither
  repeated nor consecutive.
\end{Thm}

Many proofs of this identity can be found in the literature. See for instance \cite{andrews_proofs} for an overview of some of them. The Rogers-Ramanujan identity was generalized by Gordon. The statement that we need is the following, see Theorem 7.5 from \cite{andrews_partitions}.

\begin{Thm}\label{thm:gordon}
Let $k\geq 2$. Let $B_k(n)$ denote the number of partitions of $n$ of the form $(\lambda_1,\ldots,\lambda_r)$, where $\lambda_j - \lambda_{j+k-1}\geq 2$ for all $j \in \{ 1, \ldots, r-k+1\}$. Let $A_k(n)$ denote the number of partitions of $n$ into parts which are not congruent to $0, k$ or $k+1$ modulo $2k+1$. Then $A_k(n) = B_k(n)$ for all $n$.
\end{Thm}

The analytic counterpart of Theorem~\ref{thm:rogersramanujan} can be formulated as (see Corollary 7.9 in \cite{andrews_partitions}):

\begin{Cor}[Rogers-Ramanujan identity, analytic
  form]\label{cor:rogersramanujan_analytic}
Theorem~\ref{thm:rogersramanujan} is equivalent to the identity
$$1 + \frac{t}{1-t} + \frac{t^4}{(1-t)(1-t^2)} +
\frac{t^9}{(1-t)(1-t^2)(1-t^3)} + \cdots = \prod_{\substack{i\geq 1\\i\equiv 1,4\bmod 5}}\frac{1}{(1-t^i)}.$$
\end{Cor}

The analytic analogue of Theorem~\ref{thm:gordon} is somewhat more involved and we will not formulate it here. The interested reader can find it in Andrews' book.

\section{The arc Hilbert-Poincar\'{e} series of $y^n=0$ and the Rogers-Ramanujan identities}\label{sec:section_5}

We are now going to compute the Hilbert-Poincar\'{e} series of the focussed arc algebra (at the origin) of the closed subscheme $X$ of $\mathbb{A}^1_k$ given by $y^n=0$, i.e.\ $X$ is the $n$-fold point, where $n\geq 2$. We fix $n$ and as before we denote by $F_i$ and
$f_i$, $i\in \N$, the generators of the defining ideals of $J_{\infty}(X)$ in $k[y_0,y_1,\ldots]$ and of $J_\infty^0(X)$ in $k[y_1,y_2,\ldots]$ respectively. Here we take $F_0:=y_0^n$ and $F_i:=D(F_{i-1})$ for $i\geq 1$, where $D$ is the $k$-derivation that sends $y_i$ to $y_{i+1}$ (see Proposition~\ref{prop:equations_by_deriving}). Then $f_i = F_i|_{y_0=0}$. To describe them explicitly, we need to introduce Bell polynomials. \\

Let $i\geq 1, 1\leq j\leq i$. The \textit{Bell polynomial} $B_{i,j}\in \mathbb{Z}[y_1,\ldots , y_{i-j+1}]$ is defined by the formula
\[  B_{i,j} := \sum \left(\frac{i!}{(1!)^{k_1}(2!)^{k_2}\cdots \bigl((i-j+1)!\bigr)^{k_{i-j+1}}}\right)\, \frac{y_1^{k_1}y_2^{k_2}\cdots y_{i-j+1}^{k_{i-j+1}}}{k_1!k_2!\cdots k_{i-j+1}!},    \]   
where we sum over all tuples $(k_1,k_2, \ldots , k_{i-j+1})$ of nonnegative integers such that 
\[ k_1 + k_2+ \cdots + k_{i-j+1} = j \quad \text{and} \quad k_1 + 2k_2 + \cdots + (i-j+1)k_{i-j+1} = i.  \]
Actually, the coefficient of $y_1^{k_1}\cdots y_{i-j+1}^{k_{i-j+1}}$ equals the number of possibilities to partition a set with $i$ elements into $k_1$ singletons, $k_2$ subsets with two elements, and so on. We put $B_{i,j}:=0$ if $j>i$. 
From the main result of \cite{bruschek_JetAndBell} we deduce:

\begin{Prop}\label{prop:equations_are_bell_polynomials}
We have $F_0 = y_0^n$ and for $i\geq 1$,
\[ F_i  =  \sum_{j=0}^{n-1} \frac{n!}{j!}\, B_{i,n-j}\, y_0^{j}.  \]
It follows that $f_i = 0$ if $i<n$, and for $i\geq n$,
\[f_i  =  n!\, B_{i,n}.\]
\end{Prop}

We endow $k[y_0,y_1,\ldots]$ with the following monomial ordering: for
$\alpha, \beta \in \N^{(\N)}$ we have $y^\alpha > y^\beta$ if and only
if $\wt\, \alpha > \wt\, \beta$ or, in case of equality, the last non-zero
entry of $\alpha - \beta$ is negative (i.e., a weighted reverse lexicographic ordering). The leading term of $F_i$
with respect to this ordering is determined by Proposition~\ref{prop:equations_are_bell_polynomials}:

\begin{Prop}\label{prop:leadingterm_Fi}
Let $i\geq 0$ and write $i = q n + r$ with $0\leq r < n$. The leading term of $F_i$ is 
$$\lt(F_i) = \binom{n}{r} \frac{i!}{(q!)^{n-r}\bigl((q+1)!\bigr)^r}\, y_q^{n-r} y_{q+1}^r. $$
For $i\geq n$, this is also the leading term of $f_i$.
\end{Prop}

It will turn out that these leading terms generate the leading ideal of the ideal $I=(f_i; i\geq n)$ of $k[y_1,y_2,\ldots]$, i.e., the ideal generated by the leading monomials of all polynomials in $I$. Theorem~\ref{thm:hilbert_leading} from the appendix tells us that we can deduce the arc Hilbert-Poincar\'e series from this leading ideal. For this, we need to compute a Gr\"obner basis. All results about Gr\"obner basis theory that we need are collected in the appendix as well.

\begin{Rem}
The results in the appendix are stated for polynomial rings in finitely many variables. In the proof of the next crucial lemma we will use them for countably many variables. We may do this, since we can `approximate' the arc Hilbert-Poincar\'e series according to Proposition~\ref{prop:approximate_by_jet_spaces}. We will explain this more precisely after the proof of the lemma.
\end{Rem}

\begin{Lem}\label{lem:gbbasis_nfold_point}
The leading ideal of $I=(f_i; i\geq n)$ is given by $L(I)=(\lm(f_i); i\geq
n)$. 
\end{Lem}

Before giving the proof of this lemma, we will give some concrete computations for $n=4$ to explain the ideas of the proof. By Gr\"obner basis theory it suffices to show that all $S$-polynomials on the $f_i$ reduce to zero modulo $\{f_i; i\geq n\}$, since the $S_{i,j}$ form a basis of the syzygies on the leading terms of the $f_i$ (see Proposition~\ref{prop:S_polynomials_form_basis} and Theorem~\ref{thm:criterion_for_Groebner_basis}). From Proposition~\ref{prop:leadingterm_Fi} we deduce that
\[  S(f_i,f_j) = S(F_i,F_j)|_{y_0=0}, \]
and so we can equally well show that the $S(F_i,F_j)$ reduce to zero modulo $\{F_i; i\geq 0\}$. Moreover we may restrict to those pairs $F_i,F_j$ for which the leading monomials have a nontrivial common factor by Proposition~\ref{prop:coprime_leading_monomials} (this is Step 2.1 in the proof of the lemma). Let us write the first $F_i$ down for $n=4$:
\begin{align*}
F_0 & = y_0^4 , \\
F_1 & = 4 y_0^3y_1, \\
F_2 & = 12y_0^2y_1^2 + 4 y_0^3 y_2, \\
F_3 & = 24 y_0y_1^3  + 36y_0^2 y_1 y_2 + 4y_0^3 y_3 , \\
F_4 & = 24y_1^4  + 144 y_0y_1^2y_2  + 36y_0^2y_2^2 + 48y_0^2y_1y_3 + 4 y_0^3y_4 , \\
F_5 & = 240 y_1^3 y_2 + 360y_0y_1y_2^2  + 240y_0y_1^2y_3 + 120 y_0^2y_2y_3 + 60 y_0^2y_1y_4 + 4y_0^3y_5 , \\
F_6 & =  1080 y_1^2 y_2^2  + 360 y_0 y_2^3  + 480 y_1^3 y_3  + 1440y_0y_1y_2y_3  + 120 y_0^2y_3^2  + 360 y_0y_1^2y_4 \\
    & \quad  + 180 y_0^2y_2y_4 + 72y_0^2y_1y_5           + 4 y_0^3y_6, \\
F_7 & =   2520y_1 y_2^3   + 5040 y_1^2y_2y_3  + 2520 y_0 y_2^2 y_3+ 1680 y_0y_1y_3^2  +  840 y_1^3y_4 \\
    & \quad + 2520 y_0y_1y_2y_4  + 420 y_0^2y_3y_4 + 504 y_0y_1^2y_5 + 252 y_0^2y_2y_5 + 84 y_0^2y_1y_6 + 4 y_0^3 y_7 .
\end{align*}
We may further reduce the set of $S$-polynomials that have to be checked by invoking Proposition~\ref{prop:reduction_of_syzygy_basis}. For instance, we may forget about $S(F_0,F_3)$ if we have checked that $S(F_0,F_2)$ and $S(F_2,F_3)$ reduce to zero, since $\lm(F_2)$ divides the least common multiple of $\lm(F_0)$ and $\lm(F_3)$. Similarly, using $F_1$, we may forget about $S(F_0,F_2)$. If we do this in a precise way, then we see that we only need to check the following $S$-polynomials between the above $F_i$:
\[ \left\{ \begin{array}{l} S(F_i,F_{i+1}) \text{ for } 0\leq i\leq 6, \\
S(F_1,F_7), S(F_2,F_6) , S(F_3,F_5).
\end{array} \right.  \]
This reduction is explained in Step 1 of the proof.

To see that $S(F_i,F_{i+1})$ reduces to zero, we note the following. We start from 
\begin{equation*}
 \mathcal{R} : 4y_1F_0 - y_0F_1 = 0
\end{equation*}
and we derive this relation. This gives
\begin{equation}\label{eq:S(F_1,F_2)}
 4y_2F_0 + 3y_1F_1 - y_0F_2 = 0,
 \end{equation}
or equivalently,
\[ 12 S(F_1,F_2) = -4y_2F_0. \]
This shows that $S(F_1,F_2)$ reduces to zero modulo $\{F_i; i\in \mathbb{N}\}$. Deriving (\ref{eq:S(F_1,F_2)}) once more gives
 \[ 4y_3F_0 + 7y_2F_1 + 2y_1F_2 - y_0F_3 = 0,\]
or
\[ 24S(F_2,F_3) = - 4y_3F_0 - 7y_2F_1.  \]
Hence, $S(F_2,F_3)$ reduces to zero. Similarly, by deriving the right number of times, we can prove that all $S(F_i,F_{i+1})$ reduce to zero. That is essentially Step 2.2 in the below proof.

Finally, we have to argue why $S(F_1,F_7), S(F_2,F_6)$ and $S(F_3,F_5)$ reduce to zero. That amounts to Step 2.3 in the proof of the lemma. First we derive relation $\mathcal{R}$ four times to find
\[ 4y_5F_0 + 15y_4F_1 + 20y_3F_2 +10y_2F_3 - y_0F_5 = 0.  \]
Note that $F_4$ does not appear here. This equation can be written as
\begin{equation}\label{eq:F_3F_5}
 240S(F_3,F_5) = -4y_5F_0 - 15y_4F_1 - 20y_3F_2
 \end{equation}
and this shows that $S(F_3,F_5)$ reduces to zero. Next we look at $S(F_2,F_6) = \frac{y_2^2}{12}F_2 - \frac{y_0^2}{1080}F_6$. We note that the terms of $1080S(F_2,F_6)$ appear in  
\[ 10y_2 D^3 \mathcal{R} + y_0 D^5\mathcal{R}.  \] 
From this we deduce that
\begin{align*}
 1080S(F_2,F_6) & = - 40 y_2y_4F_0 - 110 y_2 y_3 F_1 - 10 y_1y_2F_3  - 4y_0y_6F_0 - 19y_0y_5 F_1\\
  & \quad  -35 y_0y_4 F_2 -30 y_0y_3 F_3 + y_0y_1F_5. 
\end{align*}
Again $F_4$ does not appear, but this does not yet show that $S(F_2,F_6)$ reduces to zero, since 
\[ \lm(S(F_2,F_6)) = y_0^2y_1^3y_3 < y_0y_1^4y_2 = \lm(y_0y_1F_5)\]
for instance. But we do recognize $240y_1S(F_3,F_5)$ and hence we can replace this using (\ref{eq:F_3F_5}). We find
\begin{align*}
 1080S(F_2,F_6) & = (- 40 y_2y_4 +4y_1y_5 - 4y_0y_6) F_0 + (- 110 y_2 y_3   + 15y_1y_4   - 19y_0y_5) F_1\\
  & \quad  + (20y_1y_3 -35 y_0y_4) F_2 - 30 y_0y_3 F_3. 
\end{align*}
This shows that $S(F_2,F_6)$ reduces to zero. We proceed analogously for $S(F_1,F_7) = \frac{y_2^3}{4}F_1 - \frac{y_0^3}{2520}F_7$. First we look at 
\[ 90 y_2^2 D^2 \mathcal{R} + y_0^2 D^6\mathcal{R}. \]
In there we recognize $2520S(F_1,F_7), 2160y_0y_2S(F_3,F_5)$ and $2160y_1S(F_2,F_6)$. We replace the latter two and we find the following after some computations:
\begin{align*}
 2520S(F_1,F_7) & = (-360y_2^2y_3 + 80 y_1y_2y_4 -8y_1^2y_5-36y_0y_2y_5+ 8y_0y_1y_6- 4y_0^2y_7)F_0 \\
 & \quad + (220 y_1y_2 y_3 - 30y_1^2y_4  -135y_0y_2y_4+ 38y_0y_1y_5 - 23y_0^2y_6) F_1 \\
 & \quad + ( -40y_1^2y_3- 180y_0y_2y_3+70 y_0y_1y_4 - 54y_0^2y_5)F_2 \\
  & \quad  + ( 60 y_0y_1y_3- 65y_0^2y_4 )F_3 - 40y_0^2y_3F_4 .
\end{align*}
Unfortunately this does not show yet that $S(F_1,F_7)$ reduces to zero. We have:
\[ \lm(S(F_1,F_7)) = y_0^3y_1^2y_2y_3 < y_0^2y_1^4y_3 = \lm(y_0^2y_3F_4) = \lm(y_0y_1y_3F_3) = \lm(y_1^2y_3F_2).\]
This implies that $( -40y_1^2y_3 , 60 y_0y_1y_3, - 40y_0^2y_3)$ forms a homogeneous syzygy on the leading terms of $(F_2,F_3,F_4)$. We already know that a basis for these syzygies is given by $S_{2,3}$ and $S_{3,4}$. Indeed, we may compute that
\[ -40y_1^2y_3F_2 + 60 y_0y_1y_3F_3 - 40y_0^2y_3F_4 = -480y_1y_3S(F_2,F_3) + 960y_0y_3S(F_3,F_4)\]
Moreover, we explained that $S(F_2,F_3)$ and $S(F_3,F_4)$ reduce to zero modulo $\{F_i; i\geq 0\}$. Replacing their expressions in terms of the $F_i$, we conclude that $S(F_1,F_7)$ reduces to zero as well.

From this example, we see that it will be useful for the proof of Lemma~\ref{lem:gbbasis_nfold_point} to keep track of the leading monomials of the relevant $S$-polynomials.

\begin{Prop}\label{prop:lm_S_poly_consecutive}
Let $q\geq 1, 0\leq r\leq n-1$. Then
\[ \lm\bigl(S(f_{qn+r} , f_{qn+r+1})\bigr) = \left\{\begin{array}{ll} y_{q-1}y_q^{n-r-2}y_{q+1}^{r+2} & \text{if } q\geq 2, r\neq n-1, \\
y_q^2y_{q+1}^{n-2}y_{q+2} & \text{if } q\geq 2, r = n-1, \\
y_1^{n-r+1}y_2^{r-1}y_3 & \text{if } q=1, r\neq 0.
 \end{array}\right. \]
\end{Prop}

We remark that $S(f_n,f_{n+1})= 0$.

\begin{proof}
In all three cases we have written the second biggest monomial with degree $n+1$ and weight $q(n+1)+r+1$ in the variables $y_1,y_2,\ldots$. We only have to show that the monomial occurs with nonzero coefficient in $S(f_{qn+r} , f_{qn+r+1})$. Using Proposition~\ref{prop:leadingterm_Fi} we see that this $S$-polynomial is a multiple of
\[   (n-r)(q!)(qn+r+1)y_{q+1}f_{qn+r} -(r+1)\bigl((q+1)!\bigr) y_q f_{qn+r+1}. \]
Assume for instance that $q\geq 2$ and $r\neq n-1$. A computation using Proposition~\ref{prop:equations_are_bell_polynomials} gives then 
\[ \frac{(n+r+2)(n!)\bigl( (qn+r+1)!\bigr)}{\bigl((q-1)!\bigr)(q!)^{n-r-3}\bigl((q+1)!\bigr)^{r+1}\bigl((r+2)!\bigr)\bigl((n-r-2)!\bigr)} \neq 0  \]
as the coefficient of $y_{q-1}y_q^{n-r-2}y_{q+1}^{r+2}$ in the above expression. The other cases are treated similarly.
\end{proof}

\begin{Prop}\label{prop:lm_S_poly_difficult}
Let $q\geq 1, 1\leq r\leq n-1$. Then
\[ \lm\bigl(S(f_{qn+r} , f_{(q+1)n+n-r})\bigr) = \left\{\begin{array}{ll} y_{q-1}y_q^{n-r-2}y_{q+1}^{r+1}y_{q+2}^{n-r} & \text{if } q\geq 2, r\neq n-1, \\
y_{q-1}y_{q+1}^{n-2}y_{q+2}^2 & \text{if } q\geq 2, r = n-1, \\
y_1^{n-r}y_2^{r+1}y_3^{n-r-2}y_4 & \text{if } q=1, r\neq n-1, \\
y_1^{2}y_2^{n-2}y_4 & \text{if } q=1, r=n-1.
 \end{array}\right. \]
\end{Prop}

\begin{proof}
Now $S(f_{qn+r} , f_{(q+1)n+n-r})$ is a multiple of
\begin{equation}\label{difficult_S_poly}
 \frac{((q+1)n+n-r)!}{(qn+r)!} y_{q+2}^{n-r}f_{qn+r} - \frac{\bigl((q+2)!\bigr)^{n-r}}{(q!)^{n-r}}y_q^{n-r}f_{(q+1)n+n-r}. 
 \end{equation}
In the first case, we have written the second biggest monomial with degree $2n-r$, weight $(q+1)(2n-r)$, and subject to the additional condition that $y_q$ or $y_{q+2}$ has degree at least $n-r$. It only occurs in the first term of (\ref{difficult_S_poly}) due to the factor $y_{q}^{n-r-2}$. In the second case a small computation using Proposition~\ref{prop:equations_are_bell_polynomials} shows that the second biggest monomial $y_q^2y_{q+1}^{n-3}y_{q+2}^2$ does not occur in $S(f_{qn+r} , f_{(q+1)n+n-r})$ (if $n\geq 3$). We have written the third biggest, which appears with nonzero coefficient in (\ref{difficult_S_poly}). 

If $q=1$, then we can compute that no terms containing only $y_1,y_2,y_3$ occur in (\ref{difficult_S_poly}). We have written the biggest monomial containing $y_4$ of degree $2n-r$, weight $2(2n-r)$, and subject to the additional condition that $y_1$ or $y_3$ has degree at least $n-r$. It appears in (\ref{difficult_S_poly}) with nonzero coefficient.
\end{proof}

\begin{proof}[Proof of Lemma~\ref{lem:gbbasis_nfold_point}]
  We will show that the $f_i$ form a Gr\"obner basis of $I$. Using the notation of the appendix, we will first show that the $S_{i,j}$ with $n\leq i      <j$ and
  \[ \left\{ \begin{array}{lcc} j=i+1 & & \text{\textbf{or}} \\ 
  i = qn + r, j=(q+1)n + n - r \text{ for } q\geq 1, 1\leq r \leq n-1 & & \text{\textbf{or}} \\ 
  f_i\text{ and } f_j \text{ have relatively prime leading monomials } & &  \end{array} \right.  \]
  form a homogeneous basis for the syzygies on the leading terms of the $f_i$. In the second step we will show that all elements of this basis reduce to zero modulo $\{f_i; i\geq n\}$. From Theorem~\ref{thm:criterion_for_Groebner_basis} we conclude then that the $f_i$ are indeed a Gr\"obner basis.\\
  
  \textit{Step 1. The set of $S_{i,j}$ described above forms a basis of the syzygies.} 
  
  From Proposition~\ref{prop:S_polynomials_form_basis} we know already that the set of all $S_{i,j}$ forms a homogeneous basis for the syzygies on the $f_i$. If $\lm(f_i)$ and $\lm(f_j)$ are not coprime then by Proposition~\ref{prop:leadingterm_Fi} the syzygy $S_{i,j}$ is of the type $S_{qn,qn+r}$ for $q\geq 1, 0< r< n$ or $S_{qn+r,qn+s}$, where $q\geq 1, 0< r <n, r<s<2n$. 
  
  For $r$ descending from $n-1$ down to $2$ we use Proposition~\ref{prop:reduction_of_syzygy_basis} with $g_i=f_{qn}, g_j = f_{qn+r}$ and $g_k = f_{qn+r-1}$. Indeed, the least common multiple of $\lm(f_{qn})$ and $\lm(f_{qn+r})$ equals $y_q^ny_{q+1}^r$ and this is of course divisible by $\lm(g_k)=y_q^{n-r+1 }y_{q+1}^{r-1}$. So we remove the syzygies $S_{qn,qn+n-1},\ldots , S_{qn,qn+2}$ and we are still left with a basis.
  
  Next we choose $r\in \{1,\ldots, n-2\}$, and we let $s$ descend from $n$ down to $r+2$. We use again Proposition~\ref{prop:reduction_of_syzygy_basis}, now with $g_i=f_{qn+r}, g_j = f_{qn+s}$ and $g_k = f_{qn+s-1}$. The least common multiple of $\lm(f_{qn+r})$ and $\lm(f_{qn+s})$ equals $y_q^{n-r} y_{q+1}^s$ and this is divisible by $\lm(g_k)=y_q^{n-s+1} y_{q+1}^{s-1}$. For these values of $r$ and $s$ we remove the syzygies $S_{qn+r,qn+s}$ and we still have a basis.
  
  Next we let $r$ go up from 1 to $n-2$ and we choose $s\in \{ n+ 1,n+2,\ldots , 2n-r-1\}$. We take $g_i=f_{qn+r}, g_j = f_{qn+s}$ and $g_k = f_{qn+r+1}$. The least common multiple of $\lm(f_{qn+r})$ and $\lm(f_{qn+s})$ equals then $y_q^{n-r}y_{q+1}^{2n-s}y_{q+2}^{s-n}$. This is divisible by $\lm(g_k)=y_q^{n-r-1}y_{q+1}^{r+1}$. For these values of $r$ and $s$ we can again remove the syzygies $S_{qn+r,qn+s}$ and we keep a basis.
  
  Finally we choose $r\in \{2,\ldots,n-1\}$, and we let $s$ descend from $2n-1$ down to $2n-r+1$. We take $g_i=f_{qn+r}, g_j = f_{qn+s}$ and $g_k = f_{qn+s-1}$. The least common multiple of $\lm(f_{qn+r})$ and $\lm(f_{qn+s})$ equals then $y_q^{n-r}y_{q+1}^{r}y_{q+2}^{s-n}$. This is divisible by $\lm(g_k)=y_{q+1}^{2n-s+1}y_{q+2}^{s-n-1}$. For these values of $r$ and $s$ we remove once more the syzygies $S_{qn+r,qn+s}$ and we find the basis that we were looking for.\\


  
\textit{Step 2. All the elements of this basis reduce to zero modulo $\{f_i; i\geq n\}$.}

\textit{Step 2.1.} First we note that $S(f_i,f_j)$ reduces to zero modulo $\{f_i; i\geq n\}$ if $\lm(f_i)$ and $\lm(f_j)$ are relatively prime by Proposition~\ref{prop:coprime_leading_monomials}.

\vspace{3mm}

\textit{Step 2.2.}  For the other two cases we will exploit the differential structure of the ideal
  $F_0, F_1, \dots$. We have $F_i = D^i(F_0)$, where $D$ is the $k$-derivation determined by $D(y_j) = y_{j+1}$ for $j\geq 0$. Since $F_0 = y_0^n$ and $F_1 = D(y_0^n) = ny_0^{n-1}y_1$ we have the simple
  relation 
  \begin{equation}\label{eq:basicrelation}
    \mathcal{R}\colon ny_1F_0 - y_0F_1=0.
  \end{equation}
Let $q\geq 1$ and $r\in \{0,\ldots, n-1\}$. Applying $D^{q(n+1)+r}$ to the relation $\mathcal{R}$ (using the generalized Leibniz
  rule) and evaluating in $y_0=0$ yields
\begin{align*}
  0 & = - \binom{q(n+1)+r}{n-1} y_{q(n+1)+ r - n + 1} f_{n}\\
  & \quad +  \sum_{\alpha = n}^{q(n+1)+r-1} \binom{q(n+1)+r}{\alpha} \left[n y_{q(n+1)+ r - \alpha + 1} f_{\alpha} -  y_{q(n+1)+ r - \alpha} f_{\alpha+1} \right]\\
  & \quad + n y_{1} f_{q(n+1)+r} \\
    & = n\binom{q(n+1)+r}{q} y_{q+1}f_{qn+r} +
  n\binom{q(n+1)+r}{q-1}y_{q}f_{qn+r+1}\\ & \quad - \left[\binom{q(n+1)+r}{q+1} y_{q+1}f_{qn+r} +
  \binom{q(n+1)+r}{q}y_{q}f_{qn+r+1}\right] + E\\
  & = \frac{(q(n+1)+r)!\, (n-r)}{(q+1)!\, (qn+r)!} y_{q+1}f_{qn+r} - \frac{(q(n+1)+r)!\, (r+1)}{q!\, (qn+r+1)!} y_{q}f_{qn+r+1} + E,
\end{align*}
where we denote by $E$ the remaining terms in the expression of the
derivative. The polynomial $E$ is a $\mathbb{Z}$-linear combination of
$y_{q(n+1)+ r - n + 1} f_{n},\ldots, y_{q+2}f_{qn+r-1}$ and $y_{q-1}f_{qn+r+2},\ldots, y_1f_{q(n+1)+r}$. Note that
$D^{q(n+1)+r}\mathcal{R}$ has weight $q(n+1)+r+1$ and is homogeneous of degree $n+1$
with respect to the standard grading. The monomial $M=y_q^{n-r}y_{q+1}^{r+1}$ is
maximal among those monomials which are of weight $q(n+1)+r+1$ and degree
$n+1$. It cannot appear in $E$ and it is the least common multiple of the leading monomials of $f_{qn+r}$ and $f_{qn+r+1}$. Hence we conclude that 
\[ \frac{(q(n+1)+r)!\, (n-r)}{(q+1)!\, (qn+r)!} y_{q+1}f_{qn+r} - \frac{(q(n+1)+r)!\, (r+1)}{q!\, (qn+r+1)!} y_{q}f_{qn+r+1} \]
is a multiple of the $S$-polynomial $S(f_{qn+r},f_{qn+r+1})$ (this can also easily be deduced from Proposition~\ref{prop:leadingterm_Fi}).

Moreover, we have seen in the proof of Proposition~\ref{prop:lm_S_poly_consecutive} that the second biggest monomial of weight $q(n+1)+r+1$ and degree $n+1$ in $y_1,y_2,\ldots$ does occur in $S(f_{qn+r},f_{qn+r+1})$. Thus the equation
\[  \frac{(q(n+1)+r)!\, (n-r)}{(q+1)!\, (qn+r)!} y_{q+1}f_{qn+r} - \frac{(q(n+1)+r)!\, (r+1)}{q!\, (qn+r+1)!} y_{q}f_{qn+r+1} = - E \]
shows that $S(f_{qn+r},f_{qn+r+1})$ reduces to zero modulo $\{f_i; i\geq n\}$. 



\vspace{3mm}

\textit{Step 2.3.} Now let $q\geq 1, 1\leq r \leq n-1$. We are left with showing that 
\[S(f_{qn+r},f_{(q+1)n + n-r})\] 
reduces to zero modulo $\{f_i ; i\geq n\}$. We use descending induction on $r$, starting with the initial cases $r=n-1$ and $r=n-2$. For $r=n-1$ we consider the relation $\mathcal{R}$ from equation (\ref{eq:basicrelation}). Analogously to Step 2.2, we derive it $(n+1)(q+1) - 1$ times and put $y_0=0$ to find that
\begin{align*}
  0 & = - \binom{(n+1)(q+1) - 1}{n-1} y_{(n+1)(q+1) - n } f_{n}  \\
  & \quad +  \sum_{\alpha = n}^{(n+1)(q+1) - 2} \binom{(n+1)(q+1) - 1}{\alpha} \left[n y_{(n+1)(q+1) - \alpha } f_{\alpha} -  y_{(n+1)(q+1) - 1 - \alpha} f_{\alpha+1} \right]\\
  & \quad + n y_{1} f_{(n+1)(q+1) - 1} \\
    & = \frac{((q+1)(n+1)-1)!\, (n+1)}{(q+2)!\, (qn+n-1)!} y_{q+2}f_{qn+n-1}\\
    & \quad - \frac{((q+1)(n+1)-1)!\, (n+1)}{q!\, (qn+n+1)!} y_{q}f_{qn+n+1} + E,
\end{align*} 
where $E$ is a $\mathbb{Z}$-linear combination of 
\[  y_{(n+1)(q+1) - n } f_{n} , \ldots, y_{q+3}f_{qn+n-2} , y_{q+1} f_{qn+n} , y_{q-1}f_{qn+n+2} , \ldots , y_{1} f_{(n+1)(q+1) - 1}.        \]
However, the coefficient of $y_{q+1} f_{qn+n}$ in $E$ equals
\[   n \binom{(n+1)(q+1) - 1}{qn+n} - \binom{(n+1)(q+1) - 1}{qn+n-1}  = 0.  \]
It follows that
\[ \frac{((q+1)(n+1)-1)!\, (n+1)}{(q+2)!\, (qn+n-1)!} y_{q+2}f_{qn+n-1} - \frac{((q+1)(n+1)-1)!\, (n+1)}{q!\, (qn+n+1)!} y_{q}f_{qn+n+1} \]
is a multiple of $S(f_{qn+n-1},f_{qn + n + 1})$ since the monomial $y_{q}y_{q+1}^{n-1}y_{q+2}$, which is the least common multiple of the leading monomials of $f_{qn+n-1}$ and $f_{qn + n + 1}$, cannot occur in $E$. Moreover, from Proposition~\ref{prop:leadingterm_Fi} and Proposition~\ref{prop:lm_S_poly_difficult} we conclude that $S(f_{qn+n-1},f_{qn + n + 1})$ reduces to zero modulo $\{f_i; i\geq n\}$. 

\vspace{1mm}

Next consider $r=n-2$ (and $n\geq 3$). We look at the two relations
\[ \mathcal{A}_1 : \frac{(q!)^2 (q+2)!}{\bigl((q+1)(n+1) - 2\bigr)!} \, y_{q+2} \, D^{(q+1)(n+1) - 2} \mathcal{R} \, \bigr|_{y_0=0} = 0 \]
and
\[ \mathcal{A}_2 : \frac{ q! \bigl((q+2)!\bigr)^2 }{\bigl((q+1)(n+1)\bigr)!} \,  y_{q} \,  D^{(q+1)(n+1)}\mathcal{R}\, \bigr|_{y_0=0} = 0.    \]
We expand the left hand side of $\mathcal{A}_1$ as a $\mathbb{Q}$-linear combination of
\[ y_{q+2} y_{q(n+1) } f_n , \ldots , y_{q+2} y_{1} f_{q (n+1)+ n-1 },  \]
and the left hand side of $\mathcal{A}_2$ as a $\mathbb{Q}$-linear combination of
\[ y_{q} y_{q(n+1) + 2 } f_n , \ldots , y_{q} y_{1} f_{(q+1)(n+1)  }.  \]
A computation shows that the coefficient of $y_{q+2}^2f_{qn+n-2}$ in $\mathcal{A}_1$ equals
\[   \frac{(n+2) (q!)^2 }{( qn + n - 2)!} \]
and that the coefficient of $y_{q}^2f_{qn+n+2}$ in $\mathcal{A}_2$ is
\[   - \frac{ (n+2) \bigl((q+2)!\bigr)^2 }{( qn + n + 2)!}.    \] 
It follows from Proposition~\ref{prop:leadingterm_Fi} that a multiple of $S(f_{qn+n-2},f_{qn+n+2})$ occurs in the left hand side of $\mathcal{A}_1 + \mathcal{A}_2$. By a similar computation we see that the term of $\mathcal{A}_1$ containing $y_{q+1}y_{q+2}f_{qn+n-1}$ and the term of $\mathcal{A}_2$ containing $y_{q+1}y_{q}f_{qn+n+1}$ form a multiple of $y_{q+1}S(f_{qn+n-1},f_{qn+n+1})$ in $\mathcal{A}_1 + \mathcal{A}_2$. From the above we already know that we may express this as a linear combination of
\begin{center}
$ y_{q+1}y_{(q+1)(n+1) - n } f_{n} , \ldots, y_{q+1}y_{q+3}f_{qn+n-2} ,$\\
$ y_{q+1}y_{q-1}f_{qn+n+2} , \ldots , y_{q+1}y_{1} f_{(q+1)(n+1) - 1}.$   
\end{center}
Putting everything together, we conclude that we can write $S(f_{qn+n-2},f_{qn+n+2})$ as a linear combination of
\begin{center}
 $y_{q+2} y_{q(n+1) } f_n , \ldots , y_{q+2}y_{q+3}f_{qn+n-3},$\\
 $y_{q+2}y_{q}f_{qn+n},  \ldots, y_{q+2} y_{1} f_{q (n+1) + n-1 }, $\\
 $y_{q} y_{q(n+1) + 2 } f_n , \ldots , y_q y_{q+2} f_{qn+n} ,$ \\
 $y_q y_{q-1} f_{qn+n+3}  , \ldots, y_{q} y_{1} f_{(q+1)(n+1)}, $ \\
 $y_{q+1}y_{(q+1)(n+1) - n } f_{n} , \ldots, y_{q+1}y_{q+3}f_{qn+n-2} ,$\\
 $y_{q+1}y_{q-1}f_{qn+n+2} , \ldots , y_{q+1}y_{1} f_{(q+1)(n+1) - 1}.$   
\end{center}
We want to apply Propositions~\ref{prop:lm_S_poly_difficult} and \ref{prop:leadingterm_Fi} to conclude that $S(f_{qn+n-2},f_{qn+n+2})$ reduces to zero modulo $\{f_i;i\geq n\}$. The only problem is the appearance of $y_q y_{q+2} f_{qn+n}$ (twice) in the above list. However, we can compute that its coefficient in $\mathcal{A}_1 + \mathcal{A}_2$ equals zero!

\vspace{1mm}

Finally, let $r\leq n-3$ (and $n\geq 4$). We look at the relations
\[ \mathcal{A}_1 : \frac{(q!)^{n-r} (q+2)!}{\bigl(q(n+1) +r+1 \bigr)!} \, y_{q+2}^{n-r-1} \, D^{q(n+1)+r+1} \mathcal{R} \, \bigr|_{y_0=0} = 0 \]
and
\[ \mathcal{A}_2 : \frac{ q! \bigl((q+2)!\bigr)^{n-r} }{\bigl( q(n+1) + 2n - r - 1 \bigr)!} \,  y_{q}^{n-r-1} \,  D^{q(n+1) + 2n - r - 1 }\mathcal{R}\, \bigr|_{y_0=0} = 0.    \]
We expand the left hand side of $\mathcal{A}_1$ as a $\mathbb{Q}$-linear combination of
\[ y_{q+2}^{n-r-1} y_{ q(n+1) + r + 2 - n } f_n , \ldots , y_{q+2}^{n-r-1} y_{1} f_{q (n+1)+ r + 1},  \]
and the left hand side of $\mathcal{A}_2$ as a $\mathbb{Q}$-linear combination of
\[ y_{q}^{n-r-1} y_{ q(n+1) + n-r } f_n , \ldots , y_{q}^{n-r-1} y_{1} f_{ q(n+1) + 2n - r - 1  }.  \]
As before, we may check that a multiple of 
\[ S(f_{qn+r},f_{(q+1)n+n-r}) \]
occurs in the left hand side of $\mathcal{A}_1 + \mathcal{A}_2$. Similarly, multiples of
\[  y_{q+1}S(f_{qn+r+1},f_{(q+1)n+n-r-1}) \text{ and } y_qy_{q+2}S(f_{qn+r+2},f_{(q+1)n+n-r-2}) \] 
occur there. By induction, we know that the latter two $S$-polynomials reduce to zero modulo $\{f_i;i\geq n\}$ and we replace them by their expression in terms of the $f_i$. More precisely: $S(f_{qn+r},f_{(q+1)n+n-r})$ can be expressed as a linear combination of terms of the form $M f_{a}$ where $M$ is a monomial of degree $n-r$ and where $\wt\, M + a = (q+1)(2n-r)$. The maximum of $\lm(Mf_a)$ can be attained at several places. A careful analysis learns that
\[ \lm(M f_a) \leq y_{q-1}y_q^{n-r-3}y_{q+1}^{r+3}y_{q+2}^{n-r-1} =: N, \] 
where the latter monomial can occur as
\begin{center}
 $\lm(y_{q-1}y_{q+2}^{n-r-1}f_{qn+r+3}), \lm\bigl(y_{q+1}S(f_{qn+r+1},f_{(q+1)n+n-r-1})\bigr)$, \\ 
 or as $\lm\bigl(y_qy_{q+2}S(f_{qn+r+2},f_{(q+1)n+n-r-2})\bigr)$
\end{center}
by Propositions~\ref{prop:leadingterm_Fi} and \ref{prop:lm_S_poly_difficult}. Here (and from now on) we assume that $q\geq 2$. The case $q=1$ can be treated in a similar way. Only the following expressions of the form $Mf_a$ have $N$ as leading monomial:
\begin{center}
 $y_{q+2}^{n-r-1}y_{q-1}f_{qn+r+3}, y_{q-1}y_q^{n-r-3}y_{q+2}^2 f_{(q+1)n+n-r-3}$, \\
 $ y_{q-1}y_q^{n-r-3}y_{q+1}y_{q+2} f_{(q+1)n+n-r-2}, y_{q-1}y_q^{n-r-3}y_{q+1}^2f_{(q+1)n+n-r-1}$.
\end{center}
Since $\lm\bigl(S(f_{qn+r},f_{(q+1)n+n-r})\bigr)< N$ we cannot yet conclude that this $S$-polynomi\-al reduces to zero, but we see that the four expressions above must give rise to a homogeneous syzygy on the leading terms of
\[ f_{qn+r+3},f_{(q+1)n+n-r-3},f_{(q+1)n+n-r-2},f_{(q+1)n+n-r-1}.\]
From Step 1 we know that a basis for these syzygies is given by 
\begin{center}
  $S_{qn+r+3,(q+1)n+n-r-3} , S_{(q+1)n+n-r-3,(q+1)n+n-r-2}$, \\
  and $S_{(q+1)n+n-r-2,(q+1)n+n-r-1}$.
 \end{center}
By induction and by Step 2.2 we know that the corresponding $S$-polynomials reduce to zero modulo $\{f_i;i\geq n\}$. Using this, we conclude that $S(f_{qn+r},f_{(q+1)n+n-r})$ can be expressed as a linear combination of terms $Mf_a$ as above, and with $\lm(Mf_a) < N$. But we may repeat a similar argument to get rid of all monomials between
\[ \lm\bigl(S(f_{qn+r},f_{(q+1)n+n-r})\bigr) = y_{q-1}y_q^{n-r-2}y_{q+1}^{r+1}y_{q+2}^{n-r} \]
and $N$. We just have to remark that at no stage of this process the monomial
\[   y_q^{n-r} y_{q+1}^{r} y_{q+2}^{n-r}  \] 
appears as leading monomial of a term (since in all terms there are factors $y_{q-1}$, $y_{q-2},\ldots$ or $y_{q+3},y_{q+4},\ldots$ involved). This ends the proof of the lemma.
\end{proof}

\begin{Rem}\label{rem:groebner_basis_in_infinitely_many_variables}
We could have avoided to use polynomial rings in countably many variables. In fact the following holds: the leading monomials of
$(f_n,f_{n+1},\ldots,f_{n+m})$ of weight less than or equal to $n+m$ are generated by $\lm(f_i)$, $n\leq i \leq n+m$. In other words: there exists a Gr\"{o}bner basis of $(f_n, \ldots, f_{n+m})$ such that all added elements will be of weight
larger than or equal to $n+m+1$. 
\end{Rem}

\subsection{Computation of the arc Hilbert-Poincar\'{e} series of the $n$-fold point}
Using Gordon's generalization of the Rogers-Ramanujan identity (Theorem~\ref{thm:rogersramanujan}) we immediately obtain an
explicit description of the arc Hilbert-Poincar\'{e} series of the $n$-fold point by a combinatorial interpretation of the leading ideal
$L(I)$ as it was computed
in Lemma~\ref{lem:gbbasis_nfold_point}.

\begin{Thm}\label{thm:hp_nfold_point}
  The Hilbert-Poincar\'{e} series of the focussed arc algebra $J_\infty^0(X)$ of
  the $n$-fold point $X = \{y^n=0\} \subset \mathbb{A}_k^1$ over the origin equals: 
  $$\HP_{J_\infty^0(X)}(t)=\Hh\cdot
  \prod_{\substack{i\geq 1\\ i \equiv 0, n,n+1 \\ \bmod 2n+1}} (1-t^i).$$
  Equivalently, $$\HP_{J_\infty^0(X)}(t)=\prod_{\substack{i\geq 1\\ i\not\equiv 0,n,n+1\\ \bmod 2n+1}} \frac{1}{1-t^i}
  .$$
\end{Thm}

\begin{proof}
  It is a general fact from the theory of Hilbert-Poincar\'{e} series that the Hilbert-Poincar\'{e} series
  of a homogeneous ideal is precisely the Hilbert-Poincar\'{e} series of the leading
  ideal (see Theorem~\ref{thm:hilbert_leading}), i.e.,
  $$\HP_{J_\infty^0(X)}(t)= \HP_{k[y_i;i\geq 1]/L(I)}(t),$$
  where $I$ is as in Lemma~\ref{lem:gbbasis_nfold_point}. By that lemma and Proposition~\ref{prop:leadingterm_Fi} the leading ideal $L(I)$ is generated by monomials of the form $y_{q}^{n-r}y_{q+1}^r$ for $q\geq 1$ and $0\leq r \leq n-1$. Recall
  that the weight of a monomial $y^\alpha=y_{i_1}^{\alpha_1}\cdots
  y_{i_e}^{\alpha_e}$ is precisely $\alpha_1\cdot i_1+ \cdots +
  \alpha_e\cdot i_e$. Thus factoring out $L(I)$ and computing the
  Hilbert-Poincar\'{e} series of the corresponding graded algebra is equivalent to
  counting partitions $(\lambda_1,\ldots,\lambda_s)$ of natural numbers such that $ \lambda_j - \lambda_{j+n-1} > 2$ for all $j$. This is precisely what is counted in Theorem~\ref{thm:gordon}. Hence, we obtain:
  $$\HP_{J_\infty^0(X)}(t)=\prod_{\substack{i\geq 1\\ i\not\equiv 0,n,n+1\\ \bmod 2n+1}} \frac{1}{1-t^i}.$$
  The fact that the right hand side of this equation equals the
  generating series of the number of partitions of $n$ into parts
  which are not congruent to $0,n$ or $n+1$ modulo $2n+1$ is standard in the theory
  of generating series. 
\end{proof}

\subsection{An alternative approach to Rogers-Ramanujan}
In the previous section we used a combinatorial interpretation of the
leading ideal of $I=(f_n,f_{n+1},\ldots)$ to compute the Hilbert-Poincar\'e series of the
corresponding graded algebra. There are commutative algebra methods to
do this as well which yield an alternative approach to the (first) Rogers-Ramanujan
identity. Of course, we consider the case where $n=2$ here, i.e.\ the case of the double point. We will obtain a recursion formula for the
generating functions appearing in the Rogers-Ramanujan identity which has already been considered
by Andrews and Baxter in \cite{andrews_baxter}, though the present
approach gives a natural way to obtain it.\\

Consider the graded algebra $S=k[y_i;i\geq 1]/L(I)$. It is immediate
(see the proof of Theorem~\ref{thm:hp_nfold_point}) that its
Hilbert-Poincar\'e series equals the generating series of the number of partitions of
an integer $n$ without repeated or consecutive parts. Differently, we
compute the Hilbert-Poincar\'e series of $S$ by recursively defining a sequence of
formal power series (generating functions) in $t$ which converges in
the $(t)$-adic topology to the desired Hilbert-Poincar\'e series. We will simply write $k[\geq d]$ for the polynomial ring $k[y_i; i \geq d]$. It will be endowed with the grading $\wt\, y_i = i$. The ideal generated by $y^2_i, y_iy_{i+1}$ for $i\geq d$ in $k[\geq  d]$ will be denoted by $I_d$. We will still write $I_d$ for the ``same'' ideal in $k[\geq d']$ if $d'\leq d$. As usual, if $E$ is an ideal in a ring
$R$ and $f\in R$ then we denote the ideal quotient, i.e., 
\[ \{a\in R\, ;\,  a\cdot f\in E \} \]
by $(E:f)$. \\

Corollary~\ref{cor:hilbert_hilbert_trick_2} implies that $$\HP_{k[\geq
  d]/I_d}(t) = \HP_{k[\geq d]/(I_d,y_d)}(t) + t^d\cdot \HP_{k[\geq
  d]/(I_d:y_d)}(t).$$

Moreover, a quick computation shows the following.

\begin{Prop}
With the notation introduced above we have:
\begin{eqnarray*}
(I_d,y_d)  & = & (y_d,I_{d+1})\\
(I_d:y_d) & = & (y_d,y_{d+1},I_{d+2}).
\end{eqnarray*}
\end{Prop}

This immediately implies 
 $$\HP_{k[\geq  d]/I_d}(t) = \HP_{k[\geq d+1]/I_{d+1}}(t) + t^d\cdot \HP_{k[\geq
  d+2]/I_{d+2}}(t).$$

For simplicity of notation let $h(d)$ stand for $\HP_{k[\geq
  d]/I_d}(t)$. Then:
\begin{equation}\label{eq:recursion}
h(d) = h(d+1) + t^d\cdot h(d+2)
\end{equation}
and:

\begin{samepage}
\begin{Prop}\label{prop:recursion_double}
  For the Hilbert-Poincar\'e series $\HP_{J_\infty^0(X)}(t)=h(1)$ we obtain
$$h(1) = A_d\cdot h(d) + B_{d+1}\cdot h(d+1)$$
for $d\geq 1$ with $A_i,B_i\in k[[t]]$ fulfilling the following
recursion
\begin{eqnarray*}
  A_d & = & A_{d-1} + B_d\\
  B_{d+1} & = & A_{d-1}\cdot t^{d-1}
\end{eqnarray*}
with initial conditions $A_1=A_2=1$ and $B_2=0, B_3=t$.
\end{Prop}
\end{samepage}

\begin{proof}
  By the discussion above $h(1)$ equals $ h(2) + t\cdot h(3)$; hence,
  $A_1=A_2=1$ and $B_2=0$, $B_3=t$. Assume now that $$h(1) = A_d\cdot
  h(d) + B_{d+1}\cdot h(d+1)$$ holds for some $d\geq 2$. By
  equation (\ref{eq:recursion}) substituting for $h(d)$ yields
  \begin{eqnarray*}
    h(1) & = & A_d\cdot (h(d+1) + t^d\cdot h(d+2)) + B_{d+1}\cdot h(d+1)\\
    & = & (A_d + B_{d+1})\cdot h(d+1) + (A_d\cdot t^d)\cdot h(d+2)
  \end{eqnarray*}
  from which the assertion follows.
\end{proof}

If $(s_d)_{d\in \N}$ is a sequence of formal power series $s_d\in
k[[t]]$ we will denote by $\lim s_d$ its limit -- if it exists -- in
the $(t)$-adic topology. Since $\ord\, B_d \geq d-2$ it is immediate
that both $\lim A_d$ and $\lim B_d$ exist, in fact: $\lim B_d=0$ and
$$h(1) = \lim A_d.$$
The recursion from Proposition \ref{prop:recursion_double} can easily
be simplified. We obtain:

\begin{Cor}\label{cor:recursion_rogersramanujan}
  With the above introduced notation
  $\HP_{J_\infty^0(X)}(t)=\lim A_d$ where $A_d$ fulfills $$A_d = A_{d-1} + t^{d-2}\cdot A_{d-2}$$
  with initial conditions $A_1=A_2=1$.
\end{Cor}

The recursion appearing in this corollary is well-known since Andrews and Baxter \cite{andrews_baxter}. Its limit is
precisely the infinite product
\[ \prod_{\substack{i\geq 1\\ i\equiv 1, 4 \bmod 5}} \frac{1}{1-t^i}, \]
i.e., the generating series of the number of partitions with parts equal to $1$ or $4$ modulo $5$. Note, that our
construction gives the generating series $G_i$ defined in the paper by Andrews and Baxter an interpretation as
Hilbert-Poincar\'e series of the quotients $k[\geq i]/I_i$. This immediately implies that the series $G_i$ are of the
form $G_i = 1 + \sum_{j\geq i} G_{ij} t^j$ (this observation was called an `empirical hypothesis' by Andrews and Baxter).



\section{Appendix: Hilbert-Poincar\'e series and Gr\"obner bases}\label{appendix}
In this section we collect some of the basics about the theory of
Hilbert-Poincar\'{e} series. For a detailed introduction, especially
proofs, we refer to \cite{greuel_pfister}. We also recall some results on Gr\"obner basis theory from \cite{Cox_Little_Oshea}.\\

Let $A$ be a ($\Z$-)graded $k$-algebra and let $M=\oplus_{i\in \Z} M_i$ be a
graded $A$-module with $i$th graded pieces $A_i$ and $M_i$ of finite
$k$-dimension. The {\it Hilbert function $H_M\colon \Z \ra \Z$ of $M$}
is defined by $H_M(i)= \dim_k M_i$, and its corresponding generating
series $$\HP_M(t)=\sum_{i\in \Z}H_M(i)t^i \in \Z((t))$$ is
called the {\it Hilbert-Poincar\'{e} series of $M$}. It is well-known
that if $A$ is a Noetherian $k$-algebra generated by homogeneous
elements $x_1, \ldots, x_n$ of degrees $d_1,\ldots, d_n$ and $M$ is a finitely
generated $A$-module
then $$\HP_M(t)=\frac{Q_M(t)}{\prod_{i=1}^n(1-t^{d_i})}$$ for some
$Q_M(t) \in \Z[t]$ which is called the {\it (weighted) first Hilbert
 series of $M$}. If $A$ respectively $M$ is non-Noetherian then the
Hilbert-Poincar\'{e} series of $M$ need not be rational anymore. For
the rest of this section we assume that the polynomial ring
$k[x_1,\ldots, x_n]$ is graded (not necessarily standard graded). The
notions of homogeneous ideal and degree are to be understood
relative to this grading. If $M$ is graded then for any integer $d$ we
write $M(d)$ for the {\it $d$th twist of $M$}, i.e., the graded
$A$-module with $M(d)_i=M_{i+d}$.\\

The following Lemma follows immediately from additivity of dimension:  

\begin{samepage}
\begin{Lem}[Lemma 5.1.2 in \cite{greuel_pfister}]\label{lem:hilbert_trick_1}
  Let $A$ and $M$ be as above. Let $d$ be a non-negative integer,
  $f\in A_d$ and $\varphi\colon M(-d)\ra M$ be defined by $\varphi(m)
  = f\cdot m$; then $\ker(\varphi)$ and $\coker(\varphi)$ are graded
  $A/(f)$-modules with the induced gradings and $$\HP_M(t) = t^d\cdot
  \HP_M(t) + \HP_{\coker(\varphi)}(t) - t^d\cdot \HP_{\ker(\varphi)}(t).$$
\end{Lem}
\end{samepage}

As an immediate consequence we obtain the useful:

\begin{samepage}
\begin{Cor}[Lemma 5.2.2 in
  \cite{greuel_pfister}]\label{cor:hilbert_hilbert_trick_2}
  Let $I\subseteq k[x_1,\ldots, x_n]$ be a homogeneous ideal, and let
  $f\in k[x_1,\ldots, x_n]$ be a homogeneous polynomial of degree $d$
  then $$\HP_{k[x]/I}(t) = \HP_{k[x]/(I,f)}(t) + t^d\HP_{k[x]/(I:f)}(t).$$
\end{Cor}
\end{samepage}

For homogeneous ideals the leading ideal already determines the
Hilbert-Poincar\'{e} series. After fixing a monomial order, the \textit{leading ideal} $L(I)$ of an ideal $I$ in $k[x_1,\ldots,x_n]$ is defined as the (monomial) ideal generated by the leading monomials of all elements in $I$. Then one has:

\begin{samepage}
\begin{Thm}[Theorem 5.2.6 in
  \cite{greuel_pfister}]\label{thm:hilbert_leading} Let $>$ be any
  monomial ordering on $k[x_1,\ldots, x_n]$, let $I\subseteq k[x]$
  be a homogeneous ideal and denote by $L(I)$ its leading ideal with
  respect to $>$. Then $$\HP_{k[x]/I}(t) = \HP_{k[x]/L(I)}(t).$$
\end{Thm}
\end{samepage}

To compute the leading ideal one can use Gr\"obner bases. Let $I$ be an ideal in the polynomial ring $k[x_1,\ldots, x_n]$ with a fixed monomial order $<$. Then $\{g_1,\ldots, g_l\}\subset I$ is called a \textit{Gr\"obner basis} of $I$ if $L(I)$ is generated by $\{ \lm(g_i) \,;\, 1\leq i\leq l\}$, where we write $\lm$ for `leading monomial'. For $f\in k[x_1,\ldots, x_n]$ and a subset $H = \{h_1,\ldots, h_s\}$ of $k[x_1,\ldots ,x_n]$ one says that $f$ \textit{reduces to zero modulo} $H$ if 
\[  f = a_1h_1 + \cdots + a_sh_s \]
for $a_i \in k[x_1,\ldots, x_n]$, such that $\lm(f) \geq \lm(a_ih_i) $ whenever $a_ih_i\neq 0$. One writes $f \rightarrow_H 0$.

Finally we need the definition of a syzygy. Let $\mathcal{F}=(f_1,\ldots ,f_s)\in (k[x_1,\ldots,x_n])^s$. A \textit{syzygy} on the leading terms of the $f_i$ is an $s$-tuple $(h_1,\ldots,h_s) \in (k[x_1,\ldots,x_n])^s$ such that 
\[  \sum_{i=1}^s h_i\, \lt(f_i) = 0,\]
where $\lt$ stands for `leading term'. The set of syzygies $S(\mathcal{F})$ on the leading terms of $\mathcal{F}$ form a $k[x_1,\ldots,x_n]$-submodule of $(k[x_1,\ldots,x_n])^s$. A generating set of this module is called a \textit{basis}. With $F=\{f_1,\ldots, f_s\}$, we will say that a syzygy $(h_1,\ldots, h_s)\in S(\mathcal{F})$ \textit{reduces to zero modulo} $F$ if
\[  \sum_{i=1}^s h_if_i \rightarrow_F 0.  \]
If each $h_i$ consists of a single term $c_ix^{\alpha_i}$ and $x^{\alpha_i}\lm(f_i)$ is a fixed monomial $x^{\alpha}$ if $c_i\neq 0$, then the syzygy $(h_1,\ldots, h_s)$ is called \textit{homogeneous of multidegree} $\alpha$. For $i<j$ let $x^{\gamma}$ be the least common multiple of the leading monomials of $f_i$ and $f_j$. One calls
\[ S(f_i,f_j) := \frac{x^{\gamma}}{\lt(f_i)}\, f_i - \frac{x^{\gamma}}{\lt(f_j)}\, f_j \] 
the \textit{$S$-polynomial} of $f_i$ and $f_j$. It gives rise to the homogeneous syzygy
\[  S_{i,j} := \frac{x^{\gamma}}{\lt(f_i)}\, e_i - \frac{x^{\gamma}}{\lt(f_j)}\, e_j, \]
where $e_i$ and $e_j$ denote standard basis vectors of $(k[x_1,\ldots,x_n])^s$. Then we have the following results:

\begin{Prop}[Proposition 4 p.103 in \cite{Cox_Little_Oshea}]\label{prop:coprime_leading_monomials} Let $G\subset k[x_1,\ldots,x_n]$ be a finite set. Assume that $f,g\in G$ have relatively prime leading monomials. Then $S(f,g)\rightarrow_G 0$.
\end{Prop}

\begin{Prop}[Proposition 8 p.105 in
  \cite{Cox_Little_Oshea}]\label{prop:S_polynomials_form_basis} For an $s$-tuple of polynomials $(f_1,\ldots ,f_s)\in (k[x_1,\ldots,x_n])^s$ we have that the set of all $S_{i,j}$ form a homogeneous basis of the syzygies on the leading terms of the $f_i$.  
\end{Prop}

\begin{Thm}[Theorem 9 p.106 in
  \cite{Cox_Little_Oshea}]\label{thm:criterion_for_Groebner_basis} Let $\mathcal{G}=(g_1,\ldots,g_s)$ be an $s$-tuple of polynomials and let $I$ be the ideal of $k[x_1,\ldots,x_n]$ generated by $G=\{g_1,\ldots,g_s\}$. Then $G$ is a Gr\"obner basis for $I$ if and only if every element of a homogeneous basis for the syzygies $S(\mathcal{G})$ reduces to zero modulo $G$.
\end{Thm}

\begin{Prop}[Proposition 10 p.107 in
  \cite{Cox_Little_Oshea}]\label{prop:reduction_of_syzygy_basis} Let $\mathcal{G}=(g_1,\ldots,g_s)$ be an $s$-tuple of polynomials. Suppose that we have a subset $S \subset \{S_{i,j} \,;\, 1\leq i <j \leq s\}$ that is a basis of $S(\mathcal{G})$. Moreover, suppose that we have distinct elements $g_i,g_j,g_k$ such that $\lm(g_k)$ divides the least common multiple of $\lm(g_i)$ and $\lm(g_j)$. If $S_{i,k},S_{j,k}\in S$, then $S\setminus \{S_{i,j}\}$ is also a basis of $S(\mathcal{G})$. Here we put $S_{i,j} := S_{j,i}$ if $i> j$. 
\end{Prop}

\bibliographystyle{alpha}
\bibliography{./ArcsAndRogersRamanujanIdentities_bib}

\end{document}